\documentclass[10pt]{amsart}
\usepackage{amsfonts,amssymb,mathtools,amsmath, array, amsthm}
\usepackage{geometry}
\usepackage[hidelinks]{hyperref}
\usepackage{tikz}
\usepackage{textcomp, color}
\usepackage[capitalise]{cleveref}
\usepackage{orcidlink}

\theoremstyle{plain}

\newtheorem*{theoA}{Theorem A}
\newtheorem*{theoB}{Theorem B}

\newtheorem{theo}{Theorem}[section]
\newtheorem{lem}[theo]{Lemma}

\newtheorem{corollary}[theo]{ Corollary}
\newtheorem{example}[theo]{ Example}

\newtheorem{definition}[theo]{Definition}

\begin{document}
	\baselineskip13.3pt
	\title[Sylow Theory and the Nilpotency Class of Left Nilpotent Skew Braces]{Sylow Theory and the Nilpotency Class \\ of Left Nilpotent Skew Braces}
	
	\author{G\"ulİn Ercan}
	\address{Department of Mathematics, Middle East Technical University, Ankara, Turkey}
	\email{ercan@metu.edu.tr}
	
	\author{\c{S}\"ukran G\"ul}
	\address{Department of Mathematics, Middle East Technical University, Ankara, Turkey}
	\email{gsukran@metu.edu.tr}
	
	\author{\.{I}smaİl \c{S}.\ G\"ulo\u{g}lu}
	\address{Department of Mathematics, Do\u{g}u\c{s} University, \.{I}stanbul, Turkey}
	\email{iguloglu@dogus.edu.tr}
	
	\author{M.\ YASİR KIZMAZ}
	\address{Department of Mathematics, Bilkent University, Ankara, Turkey}
	\email{yasirkizmaz@bilkent.edu.tr}

	\begin{abstract}
		Let $X$ be a finite left nilpotent skew brace and  let $\pi$ be a set of primes. We show that every Hall $\pi$-subgroup of the multiplicative group $(X,\cdot)$ is a Hall $\pi$-subbrace of $X$.  This extends \cite[Theorem 11]{CDDFT} by removing the solvability assumption. As an application, we obtain an upper bound for the left nilpotency class of $X$ in terms of the left nilpotency classes of its Sylow $p$-subbraces for every prime $p$ dividing the order of $X$.
		We also show that every $p$-subbrace of $X$ is contained in some Sylow $p$-subbrace. 
	\end{abstract}
	
	
	
	\maketitle

	\section{Introduction}
	
	Motivated by earlier work of Rump \cite{rum} and of Cedó, Jespers, and Okniński \cite{CJO}, Guarnieri and Vendramin \cite{GV} introduced skew braces as a generalization of classical braces. One of the principal motivations for the development of brace theory is its close connection with set theoretic solutions of the Yang-Baxter equation. Since their introduction, skew braces have attracted considerable attention and have found applications in several areas of mathematics.
	
	A \textit{skew (left) brace} is a triple $(X,+, \cdot)$ where $(X,+)$ and $(X,\cdot)$ are groups such that the compatibility condition
	\begin{align}
	\label{compatibility}
	x\cdot (y+z)=x\cdot y-x+x\cdot z
	\end{align}
	holds for all $x,y,z \in X$. Here, the groups $(X,+)$ and $(X,\cdot)$ are respectively called the  \textit{additive group} and the \textit{multiplicative group} of the skew brace. In short, we will simply denote the skew brace $(X,+,\cdot)$ by $X$. A nonempty subset $S$ of a skew brace $X$ is said to be a \textit{subbrace} if both $(S,+)$ and $(S, \cdot)$ are subgroups of $(X,+)$ and $(X,\cdot)$, respectively.
	
	\textit{Throughout the paper, all skew braces are assumed to be finite}, $|X|$ denotes the order of the skew brace $X$ and $\pi(X)$ denotes the set of primes dividing $|X|$. Given a skew brace $X$ and $p\in \pi(X)$, a subbrace $Y$  of $X$ is called a \textit{$p$-subbrace} if $|Y|=p^k$, and a \textit{Sylow $p$-subbrace} if $|Y|=|X|_p$ where $|X|_p$ denotes the $p$-part of $|X|.$
	If $\pi$ denotes a set of primes, then a subbrace $Y$ of $X$ is called a \textit{$\pi$-subbrace} if $|Y|$ is a $\pi$-number, and a \textit{Hall $\pi$-subbrace} if its order is the largest $\pi$-number dividing $|X|$.

	One of the fundamental problems in the theory of finite skew braces is the existence of subbraces of a given order. Note that classical existence results such as Cauchy's theorem and Sylow's theorem for finite groups do not extend automatically to the brace setting. In a recent study, the question of whether Cauchy's theorem holds for skew braces was addressed by Damele and P\'erez-Calabuig ~\cite{DP}, yielding a positive result for certain skew braces. It was also shown by Caranti, Del Corso, Di Matteo, Ferrara, and Trombetti in \cite{CDDFT} that the first Sylow theorem and Hall’s theorem hold for certain families of skew braces.
	In the direction of classical existence problems, the most recent and notable result was obtained by Truman ~\cite{tru}. He proved an unconditional analogue of the first Sylow theorem for skew braces and, consequently, deduced an analogue of Cauchy's theorem. Additionally, he established an analogue of the existence assertion of Hall's theorem for skew braces with solvable additive and multiplicative groups.
	
	One of the main results in \cite{CDDFT} (see Theorem 11) is the following:
	
	\textit {
		Let $\pi$ be a set of primes, and let $X$ be a solvable skew brace which is left nilpotent. Then the Hall $\pi$-subbraces of $X$ coincide with the Hall $\pi$-subgroups of $(X, \cdot)$. In particular, every Sylow $p$-subgroup of $(X, \cdot)$ is a Sylow $p$-subbrace of $X$.
	}

	In the present paper, we extend this result to all left nilpotent skew braces by removing the solvability condition.
	We also prove an analogue of the Sylow containment theorem (every $p$-subbrace is contained in a Sylow $p$-subbrace) for left nilpotent skew braces. 

	The main results of this paper are as follows:

	\begin{theoA}
  Let $X$ be a left nilpotent skew brace and let $\pi$ be a set of primes. Then the following holds:
		\begin{enumerate}
			\item[(a)] Whenever $H$ is a Hall $\pi$-subgroup of $(X, \cdot)$, then $H$ is a Hall $\pi$-subbrace of $X$.
			\item[(b)]  Every Sylow $p$-subgroup of $(X, \cdot)$ is a Sylow $p$-subbrace of $X$.
			\item[(c)]  Every $p$-subbrace of $X$ is contained in a Sylow $p$-subbrace of $X$.
		\end{enumerate}
	\end{theoA}

	\begin{theoB}
		Let $X$ be a left nilpotent skew brace. Then
		\[
		\ell(X)\leq \max
		\Bigl\{
		\ell(P)\mid P\in \mathrm{Syl}_{p_i}(X),\; p_i\in \pi(X)
		\Bigr\}
		+1,
		\]
		where $\operatorname{Syl}_{p_i}(X)$  denotes the set of all Sylow $p_i$-subbraces of $X$.
	\end{theoB}
	
It is worth noting that Example \ref{ex:best possible bound} demonstrates that the bound in Theorem B is best possible.


	
	\section{Preliminaries}
	\label{pre}
	
	In this section, we review the basic definitions and notation regarding skew braces that will be used in the subsequent sections.

  The compatibility condition is equivalent to requiring that the multiplicative group $(X,\cdot)$ acts on the additive group $(X,+)$ by automorphisms via the \textit{lambda map}
   \[
   \lambda \colon (X,\cdot)\to\operatorname{Aut}(X,+),\qquad x \mapsto \lambda_x
   \]
    where $\lambda_x(y)=-x+x\cdot y $ for all $x, y\in X$.	 
	
    A subbrace $I$ is an \textit{ideal} of $X$ if $(I,+) \trianglelefteq (X,+)$, \ $(I,\cdot) \trianglelefteq (X,\cdot)$ and $I$ is invariant under the maps $\lambda_x$ for every $x\in X$.
     
	If $I$ is an ideal of $X$, then the set of cosets $X/I$ inherits a natural skew brace structure given by $(x+I)+(y+I) = (x+y)+I$ and $(x+I) \cdot (y+I) = (x\cdot y)+I$ for all $x, y\in X$. Furthermore, we note that additive and multiplicative cosets coincide, that is, $xI=x+I$ for all $x \in X$. 
	
	A skew brace $X$ is said to be \emph{trivial} if $x \cdot y = x+y$ for all $x, y \in X$. Notice that $\ker\lambda$ is a trivial skew brace, and hence every subgroup of $(\ker\lambda, \cdot)$ is a subbrace of $X$.
	
	We next define the \textit{star operation} $*$ on $X$ as
	\[
	x*y=-x+x\cdot y-y=\lambda_x(y)-y
	\]
	for all $x,y\in X$.

	\begin{definition}[{\cite[Definition 12.2.1]{cedoven}}]
		Let $X$ be a skew brace. Define $
		X^{1}=X,
		$
		and for $n\geq 1$,
		\[
		X^{n+1}
		=
		X*X^{n}
		=
		\langle a*x \mid a\in X,\ x\in X^{n}\rangle_{+}.
		\]
		
		The descending chain
		\[
		X^{1}\supseteq X^{2}\supseteq X^{3}\supseteq \cdots \supseteq X^{n}\supseteq \cdots
		\]
		is called the \emph{left series} of $X$.
	\end{definition}

	\begin{definition}
		A skew brace $X$ is said to be \emph{left nilpotent} if
		$X^{m}=0$
		for some positive integer $m$. In this case, the \emph{left nilpotency class} of $X$, denoted by $\ell(X)$, is the smallest positive integer $n$ such that
		\[
		X^{n+1}=0.
		\]
	\end{definition}

	Let $X$ be a skew brace. Set $G=(X,+)$ and $A=(X,\cdot)$. Then $A$ acts on $G$ by automorphisms via the lambda map. Let us denote the action of $A$ on $G$ by $a\circ b=\lambda_a(b)$ for all $a,b\in X$. We are going to use $[a,b]$ to denote $a\circ b-b$.
	
	Then we have $$a*b=a\circ b-b=[a,b]=[\lambda_a,b]$$ with this notation. Let $\varnothing \neq U \subseteq X$.
	Then
	\[
	U*X=[U,G]=[\lambda(U),G]=\langle [u,g] : u \in U, g \in G \rangle_{+}.
	\]  $\lambda(U)=\{\lambda_a : a \in U\}$. In particular,
	we have $X*X=[A,G]=[\lambda(A),G]=\langle [a,b] : a,b \in X \rangle_{+}$. Then one can observe that $$X^{n+1}=[A,\ldots,A,G]_n=[\lambda(A),\ldots,\lambda(A),G]_n$$ for all $n\geq 1$, where $[A,\ldots,A,G]_n:=[\underbrace{A,A,\ldots,A}_{n},G]$.

	\begin{lem}\label{lem:invariant}
		Let $A$ be a group acting by automorphisms on the group $G$. If $B\unlhd A$, then $[B,G]$ is an $A$-invariant normal subgroup of $G$.
	\end{lem}

	\begin{proof}
		Let $\Gamma=G\rtimes A$. Identify $G$ with the normal subgroup $G\times \{1\}$ of $\Gamma$, and $A$ with the subgroup $\{0\}\times A$ of $\Gamma$. The action of $A$ on $G$ is given by conjugation in $\Gamma$, and so $a\circ [b,g]={}^{a}[b,g]=[{}^{a}b,{}^{a}g]\in [B,G]$ for all $a\in A$, $b\in B$ and $g\in G$. Then $[B,G]=[B\times \{1\},\{0\}\times G]$ is normalized by $A$. On the other hand, $[B,G]$ is always a normal subgroup of $G$, and so $[B,G]$ is an $A$-invariant normal subgroup of $G$.
	\end{proof}

	\begin{lem}[{\cite[Lemma 12.2.5]{cedoven}}]\label{lem:direct-sum}
		Let $A$ be a skew brace such that the additive group is a direct sum of ideals
		$I_1$ and $I_2$, that is,
		\[
		A=I_1+I_2
		\qquad\text{and}\qquad
		I_1\cap I_2=\{0\}.
		\]
		Then the map
		\[
		f:A\longrightarrow I_1\times I_2
		\]
		defined by
		\[
		f(a_1+a_2)=(a_1,a_2),
		\]
		for all $a_1\in I_1$ and $a_2\in I_2$, is an isomorphism of skew braces.
	\end{lem}

	\begin{lem}\cite[Theorem 12.1.13]{cedoven} \label{p-comutator}
		Let $p$ be a prime number, and let $P$ be a finite $p$-group acting by
		automorphisms on a finite group $G$. If
		\[
		[P,\ldots,P,G]_m=\{1\}
		\]
		for some $m$, then $[P,G]$ is a $p$-group.
	\end{lem}

	\begin{lem}\cite[Theorem 12.1.11]{cedoven}\label{A-nilpotent}
		Let $A$ be a finite group that acts on a finite group $G$ by automorphisms. If
		\[
		[A,\ldots,A,G]_m=\{1\}
		\]
		for some $m$, then there exists a positive integer $n$ such that
		\[
		[[A,\ldots,A,A]_n,G]=\{1\}.
		\]
		Furthermore, if the action of $A$ on $G$ is faithful, then $A$ is nilpotent.
	\end{lem}


	\section{Main results}
	
	\begin{lem}\label{lem:ideal}
		Let $X$ be a finite skew brace. Set $G=(X,+)$ and $A=(X,\cdot)$.
		Let $P \in \operatorname{Syl}_p(A)$ such that
		\[
		\lambda(P) \unlhd \lambda(A)
		\]
		and suppose that $[P,G]$ is a $p$-subgroup of $G$. Then
		\[
		P*X=[P,G]
		\]
		is an ideal of $X$.
	\end{lem}

	\begin{proof}
		Set
		\[
		I=[P,G]=[\lambda(P),G].
		\]
		By Lemma \ref{lem:invariant}, since $\lambda(P)\unlhd \lambda(A)$, the subgroup $I$ is a $\lambda(A)$-invariant normal subgroup of $G$.
		
		We first show that $I$ is a subgroup of $A$. If $x,y\in I$, then
		\[
		xy=x+\lambda_x(y)\in I,
		\]
		because $I$ is $\lambda(A)$-invariant. Hence $I\leq A$.
		
		Since $I$ is a $p$-subgroup of $G$, it follows from \(I\leq A\) that $I$ is also a $p$-subgroup of $A$. Therefore $\lambda(I)$ is a $p$-subgroup of
		$\lambda(A)$. Moreover, since $P\in \operatorname{Syl}_p(A)$, the subgroup $\lambda(P)$ is a Sylow $p$-subgroup of $\lambda(A)$. As $\lambda(P)\unlhd \lambda(A)$, it is the unique Sylow $p$-subgroup of $\lambda(A)$. Hence
		\[
		\lambda(I)\leq \lambda(P).
		\]
		Then we get
		\[
		I*X=[\lambda(I),G]\leq [\lambda(P),G]=I.
		\]
		
		On the other hand, since $I$ is $\lambda(A)$-invariant, we have
		\[
		X*I=\langle a*i=\lambda_a(i)-i \mid a\in X,\ i \in I \rangle_{+}\subseteq I
		\]
		
		It remains to prove that $I\unlhd A$. For every $a\in A$, using $ab=a+a*b+b$, we obtain
		\[
		aI=a+a*I+I=a+I
		\]
		and
		\[
		Ia=I+I*a+a=I+a.
		\]
		Since $I\unlhd G$, we have $a+I=I+a$. Therefore $aI=Ia$ for all $a\in A$.
		Thus $I\unlhd A$. Consequently, $I$ is an ideal of $X$.
	\end{proof}

We are now ready to prove Theorem A. For a set of primes $\pi$, we will denote the set of all Hall $\pi$-subbraces of $X$ by $\operatorname{Hall}_{\pi}(X)$.

	\begin{theo} 
		\label{theo:Sylow}
		Let $X$ be a left nilpotent skew brace  and let $\pi$ be a set of primes. Then the following holds:
		\begin{enumerate}
			\item[(a)] $\operatorname{Hall}_{\pi}(X, \cdot)=\operatorname{Hall}_{\pi}(X)$.
			\item[(b)] $\operatorname{Syl}_p(X,\cdot)=\operatorname{Syl}_p(X)\neq \emptyset$ for every prime $p$ dividing $|X|$.
			\item[(c)]  Every $p$-subbrace of $X$ is contained in a Sylow $p$-subbrace of $X$.
		\end{enumerate}
	\end{theo}

	\begin{proof}
		$(a)$ If $H$ is a Hall $\pi$-subbrace of $X$, then obviously $H$ is a Hall $\pi$-subgroup of $(X, \cdot)$. We argue that the converse is also true,
		 that is, each member $H$ of $\operatorname{Hall}_\pi(X,\cdot)$ is a Hall $\pi$-subbrace of $X$. Set $G=(X,+)$ and $A=(X,\cdot)$. Let $H\in \operatorname{Hall}_\pi(A)$. We proceed by induction on $|X|.$ 
		
		Let $P\in \operatorname{Syl}_p(H)$. Note that $P$ is also a Sylow $p$-subgroup of $A$. Since $X$ is left nilpotent, there exists $n\geq 1$ such that $X^{n+1}=0$, and so
		$$[A,\ldots ,A,G]_n=1.$$ It then follows by Lemma \ref{A-nilpotent} that $A/\ker\lambda$ is nilpotent, and hence $\lambda(P) \unlhd \lambda(A)$.
		
		We have  $[P,\ldots ,P,G]_n=1$, and so $I=[P,G]$ is a $p$-subgroup of $G$ by Lemma \ref{p-comutator}. Hence, we get that $I=[P,G]$ is an ideal of $X$ by Lemma \ref{lem:ideal}. Notice that the prime $p$ is
		arbitrary so we can create an ideal $I_i=[P_i,G]$ for each $p_i\in \pi(H)$. Now we argue that we may assume that $I_i\neq 0$ for some $i$. If $I_i=0$ for all $i$, then $[\lambda(P_i),G]=1$ for all $i$. 
		It follows that $[\lambda(H),G]=1$, and so $H\leq \operatorname{ker } \lambda$. It follows that the two operations coincide on $H$, and consequently $H$ is a subbrace of $X$ as desired.
		Thus we may assume that $I=I_i\neq 0$.

		Notice that $I\leq P\leq H$ as $I$ is a normal $p$-subgroup of $A$, and so $H/I \in \operatorname{Hall}_\pi(A/I)=\operatorname{Hall}_\pi(X/I, \cdot)$. By induction applied to the quotient brace $X/I$, we have
		$H/I$ is a Hall $\pi$-subbrace of $X/I$. It follows that $H$ is a Hall $\pi$-subbrace of $X$, as required.
		
		$(b)$ This directly follows from part $(a)$ by taking $\pi=\{p\}$, and using the fact that $\operatorname{Syl}_p(X,\cdot)\neq \emptyset$ by the Sylow theorems for finite groups.
		
		$(c)$ Let $U$ be a $p$-subbrace of $X$. Since $U$ is a $p$-subgroup of $A$, there exists $P\in \operatorname{Syl}_p(A)$ such that $U\leq P$. By part $(a)$, $P$ is a Sylow $p$-subbrace of $X$ containing $U$, and the result follows.
	\end{proof}

It should be noted that in general, the solvability of the multiplicative group of a skew brace does not imply that its additive group is solvable.
	

	\begin{corollary}
		Let $X$ be a left nilpotent skew brace. If $(X,\cdot)$ is solvable, then $(X,+)$ is also solvable.
	\end{corollary}

	\begin{proof}
		Assume that $(X,\cdot)$ is solvable. Then $(X,\cdot)$ has a Hall $\pi$-subgroup $H$ for every set of primes $\pi$. By Theorem \ref{theo:Sylow}, 
		$H$ is a Hall $\pi$-subbrace of $X$. In particular, $(X,+)$ has a Hall $\pi$-subgroup for every set of primes $\pi$. Hence $(X,+)$ is solvable by Hall's solvability criterion for finite groups.

		\end{proof}

	\begin{lem}\label{lem:nilpotency-products-subbraces}
		Let $Y_1,\ldots,Y_k$ be left nilpotent skew braces. Then
		\[
		\ell\left(\prod_{i=1}^k Y_i\right)
		=
		\max_{1\leq i\leq k}\ell(Y_i).
		\]
		Moreover, if $X$ is a subbrace of a left nilpotent skew brace $Y$, then
		\[
		\ell(X)\leq \ell(Y).
		\]
	\end{lem}

	\begin{proof}
		For the direct product, the left series is computed componentwise. Hence
		\[
		\left(\prod_{i=1}^k Y_i\right)^{n}
		=
		\prod_{i=1}^k Y_i^{n}
		\]
		for every $n\geq 1$. Therefore, we have
		\[
		\ell\left(\prod_{i=1}^k Y_i\right)
		=
		\max_{1\leq i\leq k}\ell(Y_i).
		\]
		
		Now suppose that $X$ is a subbrace of $Y$. Then, by induction on $n$,
		\[
		X^{n}\leq Y^{n}
		\]
		for every $n\geq 1$. Hence, if $Y^{m}=0$, then also $X^{m}=0$. Thus,
		$
		\ell(X)\leq \ell(Y).
		$
	\end{proof}

	\begin{theo} \label{thm left nil-lenght bound}
		Let $X$ be a left nilpotent skew brace. Then
		\[
		\ell(X)\leq \max
		\Bigl\{
		\ell(P)\mid P\in \mathrm{Syl}_{p_i}(X),\; p_i\in \pi(X)
		\Bigr\}
		+1.
		\]
	\end{theo}
	
	\begin{proof}
		Set $G=(X,+)$ and $A=(X,\cdot)$ and let
		$
		\pi=\pi(A/\ker\lambda).
		$
		
		If $\pi=\varnothing$,  then $A=\ker \lambda$, and so $X$ is a trivial skew brace for which the claim is obviously true. Hence we may assume that $\pi\neq\varnothing.$
		We proceed by induction on $|\pi|$.
		Suppose first that $|\pi|=1$, that is, $\pi=\{p\}.$
		
		Let $P\in \operatorname{Syl}_p(A)$ such that $\ell(P)$ is maximal among all $P\in\operatorname{Syl}_p(A)$. Then $P\in \operatorname{Syl}_p(X)$ by Theorem \ref{theo:Sylow}. Set $J=(P,+)$.
		
		Notice that for any $Q\in \operatorname{Syl}_q(X)$ with $p\neq q$, we have $Q\leq \ker\lambda$ as $A/\ker\lambda$ is a $p$-group. It follows that
		$\ell(P)\geq 1\geq \ell(Q)$, and so $P$ has the maximum left nilpotency class among all Sylow subbraces of $X$.
		Let $\ell(P)=m$ and $\ell(X)=n$.
		Then
		\[
		[\underbrace{P,P,\ldots,P}_{m},J]=0 \ \text{and } \ [\underbrace{A,A,\ldots,A}_{n},G]=0.
		\]
		
		Note that since $\lambda(P) =\lambda(A)$, we have
		\[
		[A,G]=[\lambda(A),G]
		=[\lambda(P),G]
		=[P,G],
		\] 
		and this implies that $[A,A,\ldots,A,G]_{m+1}=[P,P,\ldots,P,G]_{m+1}=[P,P,\ldots,P,I]_{m}$
		where $I=[P,G]$. Since $I$ is a $p$-group by Lemma \ref{p-comutator}, we have $I\subseteq J$ by Lemma \ref{lem:ideal}, and so
		$$[A,A,\ldots,A,G]_{m+1}=[P,P,\ldots,P,G]_{m+1}=[P,P,\ldots,P,I]_{m}\subseteq [P,P,\ldots,P,J]_m=0.$$ Hence, we get $n\leq m+1$, as required.
		
		Now suppose that $\pi=\{p_i\mid i=1,\ldots, k\}$ and $|\pi|=k>1$.
		Next, for each $i=1,\ldots, k$, we define
		\[
		I_i=[P_i,G],
		\]
		where $P_i\in \operatorname{Syl}_{p_i}(A)$ such that $\ell(P_i)$ is maximal among all $P_i\in\operatorname{Syl}_{p_i}(A)$.
		Then each $I_i$ is a $p_i$-subgroup of $G$, and $I_i$ is an ideal by Lemma \ref{lem:ideal}. Moreover, we have $\bigcap_{i=1}^{k} I_i=0$.
		
		Hence $X$ can be embedded into
		$
		\displaystyle\prod\limits_{i=1}^{k} X/I_i 
		$ by Lemma \ref{lem:direct-sum}. Notice that
		$
		P_i/I_i \leq \ker \overline{\lambda_i},
		$
		where $\overline{\lambda_i}$ is the corresponding lambda map of
		$X/I_i$ since $I_i=[P_i,G]=P_i*X$.
		
		Hence
		\[
		\bigl|\pi\bigl((X/I_i)/\ker\overline{\lambda_i}\bigr)\bigr|
		<
		|\pi|.
		\]
		Then by the inductive argument, we have
		\[
		\ell(X/I_i)\leq \ell(\overline{Q})+1,
		\]
		where $\overline{Q}\in \operatorname{Syl}_q(\overline{X})$ for some $q\neq p_i$.
		
		Since $\overline{Q}\cong Q$, it follows that $\ell(X/I_i)\leq \ell(Q)+1$. Then we can say that, for all $i$,
		\[
		\ell(X/I_i)
		\leq
		\max
		\Bigl\{
		\ell(P)\mid P\in \mathrm{Syl}_{p_i}(X),\; p_i\in \pi(X)
		\Bigr\}
		+1.
		\]
		
		Since $X$ is isomorphic to a subbrace of $\prod_{i=1}^{k} X/I_i$, we have
		\[ \ell(X)\leq \ell\left(\prod_{i=1}^{k} X/I_i\right)= \max_{1\leq i\leq k}\ell(X/I_i)\leq
		\max
		\Bigl\{
		\ell(P)\mid P\in \mathrm{Syl}_{p_i}(X),\; p_i\in \pi(X)
		\Bigr\}
		+1
		\]
		by using Lemma \ref{lem:nilpotency-products-subbraces}. This completes the proof.
	\end{proof}

We will next show that this bound is best possible. Note that, in contrast to the notation used in the remainder of the paper, the additive group is denoted by $(X, \cdot)$ and the multiplicative group by $(X, \circ)$ in the following example.

	\begin{example}
		\label{ex:best possible bound}
We now construct a family of examples of left nilpotent skew braces that show that the bound in Theorem \ref{thm left nil-lenght bound} is sharp. 
 Let $n\geq 3$ be odd and consider the dihedral group
$D_{2n}=\langle a,b\mid a^n=b^2=1,\ bab=a^{-1}\rangle$, with $A=\langle a\rangle\cong C_n$ and $B=\langle b\rangle\cong C_2$. 
Then $D_{2n}=BA$ is an exact factorization, i.e.\ $D_{2n}=BA$ and $A\cap B=1$, so every element of $D_{2n}$ is uniquely written of the form $ba^i$ with $b\in B$ and $i\in\mathbb{Z}_n$.

This factorization determines a skew brace $X$ on the underlying set of $D_{2n}$ (see also \cite[Example 8.17]{cedoven}). The additive group is $(X,\cdot)=D_{2n}$ itself,
 i.e.\ $\cdot$ is the dihedral group operation (written multiplicatively below). The circle operation is defined component-wise on the factors $B$ and $A$:
$$
(b_1a^i)\circ(b_2a^j)=b_1b_2\,a^{i+j}, \qquad b_1,b_2\in B,\ i,j\in\mathbb{Z}_n.
$$

Since $b_1b_2\in B$ and $a^{i+j}\in A$, one can check that $(X,\circ)$ is a group and that the compatibility condition holds, so $(X,\cdot, \circ)$ is a skew brace. In fact, we have 
 $(X,\circ)\cong B\times A$ via $ba^i\mapsto(b,a^i) $, which is isomorphic to a cyclic group of order $2n$.
Recall that the corresponding lambda map $\lambda_g(x)=g^{-1}(g\circ x)$. Write $g=b_1a^i$ and $x=b_2a^j$. 
Then $$\lambda_{g}(x)=\lambda_{b_1a^i}(b_2a^j)=a^{-i}b_1(b_1a^i\circ b_2a^j)=a^{-i}b_1b_1b_2a^{i+j}=a^{-i}b_2a^j a^i=a^{-i}xa^{i}$$

for all $x\in D_{2n}$. Note that $\lambda_{ba^i}$ is simply conjugation by $a^i$ and depends only on the $A$-component of $g$. In particular we observe that $B\leq\ker\lambda$.
Consequently $X^2=X*X=[A,D_{2n}]$, and since $b$ inverts the elements of the cyclic group $A$, $[A,D_{2n}]=[A,B]=A$, that is, $X^2=A\neq0$.
Moreover every element of $\lambda(X)$ acts by conjugation by an element of the abelian group $A$, so $X^3=X*X^2=[A,A]=0$. Thus $\ell(X)=2$.

Every Sylow subbrace of $X$ is trivial as a skew brace: for $p\mid n$, the Sylow $p$-subgroup $P\leq A$ satisfies $P^2=P*P=0$ as $(A,\cdot, \circ)$ is a trivial brace, while the Sylow $2$-subbrace $B$ satisfies $B^2=0$ because $B\leq\ker\lambda$. Hence $\ell(P)=1$ for every $P\in\operatorname{Syl}_{p_i}(X)$, $p_i\in\pi(X)$, so
\[
\ell(X)=2=\max\bigl\{\ell(P)\mid P\in\operatorname{Syl}_{p_i}(X),\;p_i\in\pi(X)\bigr\}+1,
\]
showing that the bound in Theorem \ref{thm left nil-lenght bound} is sharp.
\end{example}


\end{document}